\documentclass[a4paper, 11pt]{amsart}   
\usepackage{mathptmx, amssymb,amscd,latexsym, eulervm}   
\usepackage{amsmath}
\usepackage{amsthm}
\usepackage{mathdots}
\usepackage[colorlinks=true,linkcolor=blue,urlcolor=blue]{hyperref}
\usepackage{color}
\usepackage[onehalfspacing]{setspace}
\usepackage{tabularx}
\usepackage{amsfonts}
\usepackage{paralist}
\usepackage{aliascnt}
\usepackage[initials, lite]{amsrefs}
\usepackage{amscd}
\usepackage{blkarray}

\usepackage{mathbbol}
\usepackage{setspace}
\usepackage[inner=2.5cm,outer=2.5cm, bottom=3cm]{geometry}
\usepackage{calligra,mathrsfs}
\usepackage{tikz, tikz-cd}
\usetikzlibrary{patterns}

\usepackage{tikz}
\usetikzlibrary{matrix}
\usetikzlibrary{arrows,calc}
\allowdisplaybreaks

\BibSpec{collection.article}{%
	+{}  {\PrintAuthors}                {author}
	+{,} { \textit}                     {title}
	+{.} { }                            {part}
	+{:} { \textit}                     {subtitle}
	+{,} { \PrintContributions}         {contribution}
	+{,} { \PrintConference}            {conference}
	+{}  {\PrintBook}                   {book}
	+{,} { }                            {booktitle}
	+{,} { }                            {series}
	+{, vol.} { }                            {volume}
	+{,} { }                            {publisher}
	+{,} { \PrintDateB}                 {date}
	+{,} { pp.~}                        {pages}
	+{,} { }                            {status}
	+{,} { \PrintDOI}                   {doi}
	+{,} { available at \eprint}        {eprint}
	+{}  { \parenthesize}               {language}
	+{}  { \PrintTranslation}           {translation}
	+{;} { \PrintReprint}               {reprint}
	+{.} { }                            {note}
	+{.} {}                             {transition}
	+{}  {\SentenceSpace \PrintReviews} {review}
}
\AtBeginDocument{%
	\def\MR#1{}
}

\newcommand{\kk}{\mathbb{k}}

\newcommand{\NN}{\normalfont\mathbb{N}}
\newcommand{\ZZ}{\normalfont\mathbb{Z}}

\newcommand{\PP}{{\normalfont\mathbb{P}}}
\newcommand{\xx}{{\normalfont\mathbf{x}}}

\newcommand{\init}{{\normalfont\text{in}}}

\newcommand{\bn}{{\normalfont\mathbf{n}}}
\newcommand{\supp}{{\normalfont\text{supp}}}

\newcommand{\bm}{{\normalfont\mathbf{m}}}
\newcommand{\ttt}{{\normalfont\mathbf{t}}}
\newcommand{\sss}{{\normalfont\mathbf{s}}}

\newcommand{\fJ}{{\mathfrak{J}}}

\newcommand{\fS}{{\mathfrak{S}}}

\newcommand{\rank}{\normalfont\text{rank}}

\newcommand{\codim}{{\normalfont\text{codim}}}

\newcommand{\Tor}{{\normalfont\text{Tor}}}

\newcommand{\ee}{{\normalfont\mathbf{e}}}

\newcommand{\multProj}{\normalfont\text{MultiProj}}

\def\f0{\mathbf{0}}

\def\fs{\mathbf{s}}

\def\fc{\mathbf{c}}

\def\ft{\mathbf{t}}
\def\fr{\mathbf{r}}

\def\1{\mathbf{1}}

\newtheorem{theorem}{Theorem}[section]

\newtheorem{headthm}{Theorem}

\newaliascnt{headcor}{headthm}

\aliascntresetthe{headcor}

\newaliascnt{headconj}{headthm}

\aliascntresetthe{headconj}

\newaliascnt{corollary}{theorem}

\aliascntresetthe{corollary}

\newaliascnt{claim}{theorem}

\aliascntresetthe{claim}

\newaliascnt{lemma}{theorem}
\newtheorem{lemma}[lemma]{Lemma}
\aliascntresetthe{lemma}

\newaliascnt{conjecture}{theorem}
\newtheorem{conjecture}[conjecture]{Conjecture}
\aliascntresetthe{conjecture}

\newaliascnt{proposition}{theorem}
\newtheorem{proposition}[proposition]{Proposition}
\aliascntresetthe{proposition}

\theoremstyle{definition}
\newaliascnt{definition}{theorem}
\newtheorem{definition}[definition]{Definition}
\aliascntresetthe{definition}

\newaliascnt{notation}{theorem}

\aliascntresetthe{notation}

\newaliascnt{example}{theorem}
\newtheorem{example}[example]{Example}
\aliascntresetthe{example}

\newaliascnt{examples}{theorem}

\aliascntresetthe{examples}

\newaliascnt{remark}{theorem}
\newtheorem{remark}[remark]{Remark}
\aliascntresetthe{remark}

\newaliascnt{question}{theorem}

\aliascntresetthe{question}

\newaliascnt{questions}{theorem}

\aliascntresetthe{questions}

\newaliascnt{problem}{theorem}

\aliascntresetthe{problem}

\newaliascnt{construction}{theorem}

\aliascntresetthe{construction}

\newaliascnt{setup}{theorem}
\newtheorem{setup}[setup]{Setup}
\aliascntresetthe{setup}

\newaliascnt{algorithm}{theorem}

\aliascntresetthe{algorithm}

\newaliascnt{observation}{theorem}

\aliascntresetthe{observation}

\newaliascnt{defprop}{theorem}

\aliascntresetthe{defprop}

\def\equationautorefname~#1\null{(#1)\null}
\def\sectionautorefname~#1\null{Section #1\null}
\def\subsectionautorefname~#1\null{\S #1\null}

\title{Double Schubert polynomials do have saturated Newton polytopes}

\author{Federico Castillo}
\address[Castillo]{Departamento de Matem\'aticas, Pontificia Universidad Cat\'olica de Chile, Santiago, Chile}
\email{federico.castillo@mat.uc.cl}

\author{Yairon Cid-Ruiz}
\address[Cid-Ruiz]
{Department of Mathematics, KU Leuven, Celestijnenlaan 200B, 3001 Leuven, Belgium}
\email{yairon.cidruiz@kuleuven.be}

\author{Fatemeh Mohammadi}
\address[Mohammadi]
{Department of Mathematics and Computer Science, KU Leuven, 3001 Leuven, Belgium \\
Department of Mathematics and Statistics, UiT - The Arctic University of Norway, Troms\o, Norway}
\email{fatemeh.mohammadi@kuleuven.be}

\author{Jonathan Monta\~no}
\address[Monta\~no]{School of Mathematical and Statistical Sciences, Arizona State University, P.O. Box 871804, Tempe, AZ 85287-18041, United States of America}
\email{montano@asu.edu}


\begin{document}

	\dedicatory{Dedicated to Bernd Sturmfels, on the occasion of his 60th birthday, for his far-reaching contributions and impact on the careers of many young mathematicians.}
	
	\maketitle
	
	\vspace*{-.7cm}	
	\begin{abstract}
		We prove that double Schubert polynomials have the Saturated Newton Polytope property. This settles a conjecture by Monical, Tokcan and Yong. 
		Our ideas are motivated by the theory of multidegrees.
		We introduce a notion of standardization of ideals that enables us to study non-standard multigradings. 
		This allows us to show that the support of the multidegree polynomial of each Cohen-Macaulay prime ideal in a non-standard multigrading, and in particular, that of each Schubert determinantal ideal is a discrete polymatroid. 
	\end{abstract}
	
	\vspace*{-.3cm}
	\section{Introduction}
	
	Schubert polynomials are classical and important objects in algebraic combinatorics.
	They were introduced by Lascoux and Sch\"utzenberger \cite{SCUBERT_POLY_L_S} to study the cohomology classes of Schubert varieties.
	Since then, Schubert polynomials have played a fundamental role in algebraic combinatorics (see,~e.g.,~\cite{fink2018schubert, monical2019newton,BERGERON_BILLEY,BILLEY_STANLEY, KNUTSON_MILLER_SUBWORD, KNUTSON_MILLER_SCHUBERT, LLS} and the references therein).

	\medskip
	
	We first recall the definition of Schubert polynomials.
	Let $\mathscr{S}_p$ be the symmetric group on the set $[p] = \{1, \ldots, p\}$. 
	For every $i\in[p-1] = \{1,\ldots,p-1\}$ we have the transposition $\sigma_i=(i,i+1)\in \mathscr{S}_p$. Recall that the set $\mathfrak{T} = \{ \sigma_i \mid 1\leq i< p \}$ generates $\mathscr{S}_p$. 
	The \emph{length} $\ell(\pi)$ of a permutation $\pi$ is the least amount of elements in $\mathfrak{T}$ counting repetitions needed to obtain $\pi$ from the identity permutation. 
	The permutation $\pi_0=(p,p-1,\ldots,2,1)$ (in one-line notation) is the
	longest permutation and has length $\frac{p(p-1)}{2}$.
	We follow the notation of \cite{KNUTSON_MILLER_SCHUBERT} and \cite[Chapter 15]{miller2005combinatorial} to present permutations and Schubert polynomials.
	
	\begin{definition}
		The {\it double Schubert polynomial} $\fS_\pi(\ttt, \sss)  \in\ZZ[t_1,\ldots,t_p,s_1,\dots,s_p]$ of a permutation $\pi \in \mathscr{S}_p$ is defined recursively in the following way. 
		First we define $\fS_{\pi_0}=\prod_{i+j\leq p}(t_i-s_j)$, and for any permutation $\pi$ and transposition $\sigma_i$ with $\ell(\sigma_i\pi)<\ell(\pi)$ we set
		\[
		\fS_{\sigma_i\pi} \,=\, \dfrac{\fS_{\pi}-\sigma_i\fS_{\pi}}{t_i-t_{i+1}},
		\]
		where $\mathscr{S}_p$ acts only on $\ZZ[t_1,\ldots,t_p]$ by permutation of variables. 
		The \emph{(ordinary) Schubert polynomial} $\fS_\pi(\ttt, \mathbf{0}) \in\ZZ[t_1,\ldots,t_p]$ is obtained from  $\fS_\pi$ by setting each variable $s_j$ equal to $0$.
	\end{definition}
	The monomial expansion of ordinary Schubert polynomials has been combinatorially analyzed using different objects such as compatible sequences \cite{BILLEY_JOCKUSH_STANLEY}, reduced pipe dreams \cite{fomin1994grothendieck,BERGERON_BILLEY}, and Kohnert diagrams \cite{KOHNERT}.
	The description using pipe dreams also works for the double Schubert polynomials \cite[Corollary 16.30]{miller2005combinatorial}.
	We  also have a formula for double Schubert polynomials using bumpless pipe dreams \cite{LLS}.
	
	Following \cite{monical2019newton}, we say that a polynomial $f = \sum_{\bn}c_{\bn}\xx^\bn\in \ZZ[x_1,\ldots,x_n]$ has the \emph{Saturated Newton Polytope property} (SNP property for short) if the support  $\supp(f)=\{\bn\in\NN^n\mid c_\bn \neq 0\}$ of $f$ is equal to $\text{Newton}(f)\cap\NN^n$, where $\text{Newton}(f) = \text{ConvexHull}\{\bn\in\NN^n\mid c_\bn \neq 0\}$ denotes the \emph{Newton polytope} of $f$; in other words, if the support of $f$ consists of the integer points of a polytope.
	
	The main goal of this paper is to confirm the following challenging conjecture by Monical, Tokcan and Yong that appeared in \cite[Conjecture 5.2]{monical2019newton}.
	\begin{conjecture}[\cite{monical2019newton}]
		\label{conjecture}
		Double Schubert polynomials have the Saturated Newton Polytope property.
	\end{conjecture}
	We confirm the conjecture by proving a stronger result that the support of each double Schubert polynomial is a discrete polymatroid. 
	A \textit{discrete polymatroid} $\mathcal{P}$ on $[n]=\{1,\ldots,n\}$ is a collection of points in $\NN^n$ of the following form
	\[
	\mathcal{P=}\Big\{(x_1,\ldots, x_n)\in \NN^n \;\mid\; \sum_{j\in\fJ} x_j\leq r(\fJ), \;\forall \fJ\subsetneq [n], \;\sum_{i\in [n]} x_i=r([n])\Big\}
	\]
	with $r$ 
	being a rank function on $[n]$. A \emph{rank function on} $[n]$ 
	is a function $r : 2^{[n]}\rightarrow \NN$ satisfying the following three properties: (i) $r(\emptyset) = 0$, (ii) $r(\fJ_1)\leq r(\fJ_2)$ if $\fJ_1\subseteq \fJ_2 \subseteq [n]$, and (iii)  $r(\fJ_1\cap \fJ_2)+r(\fJ_1\cup \fJ_2)\leq r(\fJ_1)+r(\fJ_2)$ if $\fJ_1,\fJ_2 \subseteq [n]$.

	The following is the main theorem of this article.

	\begin{headthm}\label{thmA}
		Let $\pi \in \mathscr{S}_p$ be a permutation and $\fS_\pi(\ttt, \sss)  \in\ZZ[t_1,\ldots,t_p,s_1,\dots,s_p]$ be the corresponding double Schubert polynomial.
		Then, the support $\supp(\fS_\pi) \subset \NN^{2p}$ of $\fS_\pi$ is a discrete polymatroid on $[2p]=\{1,\ldots,2p\}$.
		In particular, the statement of \autoref{conjecture} holds.
	\end{headthm}
	
	Our approach to prove \autoref{thmA} can be summarized in the following quote by Miller and Sturmfels \cite[Introduction to Chapter 15]{miller2005combinatorial}:
	\emph{``We consider the finest possible multigrading, which demands the refined toolkit
		of a new generation of combinatorialists.''}
	More precisely, we utilize the result that double Schubert polynomials equal the multidegree polynomial of Schubert determinantal ideals with the aforementioned \emph{``finest possible multigrading''} (see \cite[Theorem 15.40]{miller2005combinatorial}), and then we develop a method of \emph{standardization of ideals}.
	This process of standardization allows us to study multidegrees in certain non-standard multigradings by reducing the problem to a standard multigraded setting.
	Our main tool is \autoref{thm_pos_multdeg} from \cite{castillo2020multidegrees} which shows that the support of the multidegree polynomial of any multihomogeneous prime ideal (with usual standard multigrading) is a discrete polymatroid. 
	Here we extend this theorem to the family of non-standard multigradings that we study. 
	
	Much interest has been paid to the important conjectures proposed in \cite{monical2019newton} and a number of them have already been confirmed (see \cite{fink2018schubert}).
	Therefore, \autoref{thmA} settles a hitherto remaining conjecture from \cite{monical2019newton} and gives further evidence to the ubiquity of the SNP property in many ``combinatorially defined polynomials''. \autoref{thmA} also gives more evidence for the presence of the {\it Lorentzian} property in double Schubert polynomials as conjectured in \cite{huh2022logarithmic}.

	The structure of the paper is as follows. We review the notion of multidegrees in \autoref{sect_recap_multdeg} and recall the connection between double Schubert polynomials and Schubert determinantal ideals in \autoref{sec:det_ideals}.
	\autoref{sec:main} contains our main results, in particular the proof of \autoref{thmA}.
	
	\medskip
	
	\noindent
	\textbf{Acknowledgments.} 
	F.C. thanks Ghent University for their hospitality. 
	F.C. was partially supported by  FONDECYT Grant 1221133.
	Y.C.R. was partially supported by an FWO Postdoctoral Fellowship (1220122N). 
	F.M. was partially supported by 
	FWO grants (G023721N, G0F5921N), the KU Leuven iBOF/23/064 grant, and the UiT Aurora MASCOT project. 
	J.M. was partially supported by NSF Grant DMS \#2001645/2303605.
	We thank the reviewer for helpful comments and suggestions.

	\section{A short recap on multidegrees}
	\label{sect_recap_multdeg}
	
	In this short section, we briefly recall the notion of multidegrees and some of its basic properties; for more details the reader is referred to \cite{miller2005combinatorial, cidruiz2021mixed}.
	
	Let $\kk$ be a field and $R = \kk[x_1,\ldots,x_n]$ be a $\ZZ^p$-graded polynomial ring (for now, we do not assume the grading to be positive).  
	Let $M$ be a finitely generated $\ZZ^p$-graded module and $F_\bullet$ be a $\ZZ^p$-graded free $R$-resolution 
	$
	F_\bullet : \; \cdots \rightarrow F_i \rightarrow F_{i-1} \rightarrow \cdots \rightarrow F_1 \rightarrow F_0
	$
	of $M$.
	Let $t_1,\ldots,t_p$ be variables over $\ZZ$ and consider the polynomial ring $\ZZ[\ttt] = \ZZ[t_1,\ldots,t_p]$, where the variable $t_i$ corresponds with the $i$-th elementary vector $\ee_i \in \ZZ^p$. 
	If we write $F_i = \bigoplus_{j} R(-\mathbf{b}_{i,j})$ with $\mathbf{b}_{i,j} = (\mathbf{b}_{i,j,1},\ldots,\mathbf{b}_{i,j,p}) \in \ZZ^p$, then we define the Laurent polynomial 
	$
	\left[F_i\right]_\ttt \, := \, \sum_{j} \ttt^{\mathbf{b}_{i,j}} = \sum_{j} t_1^{\mathbf{b}_{i,j,1}} \cdots t_p^{\mathbf{b}_{i,j,p}}.
	$
	Then, the \emph{K-polynomial} of $M$ is defined by 
	$$
	\mathcal{K}(M;\ttt) \, := \, \sum_{i} {(-1)}^i \left[ F_i \right]_\ttt.
	$$
	It turns out that, even if the grading of $R$ is non-positive and we do not have a well-defined notion of Hilbert series, the above definition of K-polynomial is an invariant of the module $M$ and it does not depend on the chosen free $R$-resolution $F_\bullet$ (see \cite[Theorem 8.34]{miller2005combinatorial}).

	\begin{definition}
		The \emph{multidegree polynomial} of a finitely generated $\ZZ^p$-graded $R$-module $M$ is the homogeneous polynomial $\mathcal{C}(M; \ttt) \in \ZZ[\ttt]$ given as the sum of all terms in 
		$$
		\mathcal{K}(M; \mathbf{1} - \ttt) = \mathcal{K}(M; 1-t_1,\ldots,1-t_p) 
		$$
		having total degree $\codim(M) = n - \dim(M)$.
	\end{definition}

	One case of particular interest is when $R$ is a standard multigraded polynomial ring. 
	We say that $R$ is \emph{standard $\ZZ^p$-graded} if the total degree of each variable $x_i$ is equal to one (i.e.,~for each $1 \le i \le n$, we have $\deg(x_i) = \ee_{k_i} \in \ZZ^p$ with $1 \le k_i \le p$).
	The study of standard multigraded algebras is of utmost importance as they correspond with closed subschemes of a product of projective spaces (see, e.g., \cite{castillo2020multidegrees} and the references therein).
	Since the coefficients of the multidegree polynomial are non-negative in the standard multigraded case, it becomes natural to address the positivity of these coefficients. 
	For each subset $\mathfrak{J} = \{j_1,\ldots,j_k\} \subseteq [p] = \{1, \ldots, p\}$ denote by $R_{(\fJ)}$ the $\ZZ^k$-graded $\kk$-algebra given by 
	$$
	R_{(\fJ)} := \bigoplus_{\substack{i_1\ge 0,\ldots, i_p\ge 0\\ i_{j} = 0 \text{ if } j \not\in \fJ}} {\left[R\right]}_{(i_1,\ldots,i_p)},
	$$
	and for any $R$-homogeneous ideal $I \subset R$ we define $I_{(\fJ)}$ as the contraction $I_{(\fJ)} := I \cap R_{(\fJ)}$.	
	The following theorem completely characterizes the positivity of multidegrees and is our main tool to prove \autoref{thmA}.

	\begin{theorem}[\cite{castillo2020multidegrees}]
		\label{thm_pos_multdeg}
		Let $R=\kk[x_1,\ldots,x_n]$ be a standard $\ZZ^p$-graded polynomial ring.
		Let $I \subset R$ be an $R$-homogeneous prime ideal.  
		Write the multidegree polynomial of $\mathcal{C}(R/I;\ttt)$ as
		$$
		\mathcal{C}(R/I;\ttt) \,=\, \sum_{\substack{\bn \in \NN^p \\ |\bn| = \codim(I)}} c_{\bn} \ttt^\bn \; \in \; \NN[t_1,\ldots, t_p].
		$$
		Then, for all $\bn = (n_1,\ldots,n_p) \in \NN^p$ with $|\bn| = \codim(I) = n - \dim(R/I)$, we have that $c_\bn > 0$ if and only if for each $\fJ = \{j_1,\ldots,j_k\} \subseteq [p]$ the inequality 
		$
		n_{j_1} + \cdots + n_{j_k}  \, \ge \, \codim\left(I_{(\fJ)}\right)
		$
		holds.
		Furthermore, the support of $\mathcal{C}(R/I; \ttt)$ is a discrete  polymatroid.
	\end{theorem}
	\begin{proof}
		Consider the  standard $\ZZ^{p}$-graded polynomial ring  $R' = R[x_{n+1},\ldots,x_{n+p}]$ with $\deg(x_{n+i}) = \ee_i \in \ZZ^p$ and notice that $\mathcal{C}(R/I;\ttt) = \mathcal{C}(R'/IR';\ttt)$. 
		Thus we assume that $I$ is a relevant prime (i.e., $I \not\supset \bigoplus_{i_1\ge 1,\ldots, i_p\ge 1} {\left[R\right]}_{(i_1,\ldots,i_p)}$), and so $\multProj(R/I) \neq \emptyset$.
		We embed $X = \multProj(R/I)$ as a closed subscheme of a multiprojective space $\PP:=\PP_\kk^{m_1} \times_\kk \cdots \times_\kk \PP_\kk^{m_p}$.
		From \cite[Remark 2.9]{castillo2020multidegrees} we have that $\bn  \in \supp(\mathcal{C}(R/I;\ttt))$ if and only if $\deg_\PP^{\bm-\bn}(X)>0$ where $\bm - \bn=(m_1-n_1,\ldots,m_p-n_p)$.
		Then, \cite[Theorem A]{castillo2020multidegrees} implies that $\bn \in \supp(\mathcal{C}(R/I;\ttt))$ if and only if $|\bn| = \codim(I)$ and $\sum_{j \in \fJ }n_{j} \ge  \codim\left(I_{(\fJ)}\right)$ for each $\fJ \subseteq [p]$. 
		Equivalently, we obtain that $\bn \in \supp(\mathcal{C}(R/I;\ttt))$ if and only if $|\bn| = \codim(I)$ and 
		$$
		\sum_{j \in \fJ }n_{j} \,\le\, \codim(I) -  \codim\left(I_{([p] \setminus \fJ)}\right) \,=\, \sum_{j \in \fJ} m_j + r([p] \setminus \fJ) - r([p])
		$$ 
		for each $\fJ  \subseteq [p]$, where $r : 2^{[p]} \rightarrow \NN$ is the rank function $r(\fJ) := \dim\big(\multProj\big(R_{(\fJ)}/I_{(\fJ)}\big)\big)$ (see \cite[Proposition 5.1]{castillo2020multidegrees}).
		Finally, we can check that $s : 2^{[p]} \rightarrow \NN$ with $s(\fJ) := \sum_{j \in \fJ} m_j + r([p] \setminus \fJ) - r([p])$ is a rank function (see, e.g., \cite[\S 44.6f]{schrijver2003combinatorial}), and so it follows that $\supp(\mathcal{C}(R/I;\ttt))$ is a polymatroid.
	\end{proof}

	\section{Schubert determinantal ideals}
	\label{sec:det_ideals}
	Here we recall the connection between double Schubert polynomials and Schubert determinantal ideals (for more details, the reader is referred to \cite[Chapters~15,~16]{miller2005combinatorial}). 
	
	First, we define matrix Schubert varieties and Schubert determinantal ideals by following \cite[Chapter 15]{miller2005combinatorial}. 
	Let $\kk$ be an algebraically closed field and  $M_p(\kk)$ be the $\kk$-vector space of $p\times p$ matrices with entries in $\kk$. 
	As an affine variety we define its coordinate ring as $\widetilde{R} =\kk[x_{i,j} \mid (i,j)\in[p]\times[p]]$. 
	Furthermore, we consider a $(\ZZ^p\oplus \ZZ^p)$-grading on $\widetilde{R}$ by setting $\deg(x_{i,j}) = \ee_i \oplus -\ee_j \in \ZZ^p \oplus \ZZ^p$, where $\ee_i \in \ZZ^p$ denotes the $i$-th elementary vector.

	\begin{definition}[{see \cite[Chapter 15]{miller2005combinatorial}}] 
		Let $\pi$ be a permutation matrix. 
		The \emph{matrix Schubert variety} $\overline{X_\pi}\subset M_p(\kk)$ is the subvariety given by
		$
		\overline{X_\pi}=\{Z\in M_p(\kk)\mid \rank(Z_{m\times n})\leq \rank(\pi_{m\times n}) \; \text{ for all }\; m,n\},
		$
		where $Z_{m\times n}$ is the restriction to the first $m$ rows and $n$ columns. 
		The \emph{Schubert determinantal ideal} $I_\pi \subset \widetilde{R}$ is the $\widetilde{R}$-homogeneous ideal generated by all minors in $\mathbf{X}_{m \times n}$ of size $1 + \rank(\pi_{m \times n})$ for all  $m$ and $n$, where $\mathbf{X} = (x_{i,j})$ is the $p\times p$ matrix with the variables of $\widetilde{R}$. 
	\end{definition}
	The following theorem collects several results of fundamental importance to our approach.
	In particular, it shows that double Schubert polynomials equal the multidegree polynomial of matrix Schubert varieties.
	To define multidegrees over $\widetilde{R}$ with its $(\ZZ^p \oplus \ZZ^p)$-grading, we consider the polynomial ring $\ZZ[\ttt,\sss] = \ZZ[t_1,\ldots,t_p,s_1,\ldots,s_p]$, where $t_i$ has degree $\ee_i \oplus \mathbf{0} \in \ZZ^p \oplus \ZZ^p$ and $s_i$ has degree $\mathbf{0} \oplus \ee_i \in \ZZ^p \oplus \ZZ^p$.
	
	\begin{theorem}
		\label{thm_Schubert}
		Let $\pi \in \mathscr{S}_p$ be a permutation and denote also by $\pi$ the corresponding permutation matrix.
		Then, the following statements hold: 
		\begin{enumerate}[\rm (i)]
			\item $I_\pi$ is a prime ideal, and so it coincides with the ideal $I(\overline{X_\pi})$ of polynomials vanishing on the matrix Schubert variety $\overline{X_\pi}$. \hfill{{\rm(}\cite{FULTON_SCHUBERT}, \cite[Corollary 16.29]{miller2005combinatorial}{\rm)}}
			\item $\widetilde{R}/I_\pi$ is a Cohen-Macaulay ring. \hfill{{\rm(}\cite{FULTON_SCHUBERT}, \cite[Corollary 16.44]{miller2005combinatorial}{\rm)}}
			\item $\fS_\pi(\ttt, \mathbf{s}) = \mathcal{C}(\widetilde{R}/I_\pi; \ttt, \sss)$.\hfill {{\rm(}\cite{FEHER_RIMANYI},\cite{KNUTSON_MILLER_SCHUBERT}, \cite[Theorem 15.40]{miller2005combinatorial}{\rm)}}
		\end{enumerate}
	\end{theorem}
	
	The next technical lemma will allow us to substitute the  grading of $\widetilde{R}$ which has negative components for the degrees of the variables.  
	Let $R =\kk[x_{i,j} \mid (i,j) \in [p]\times[p]]$ with induced $(\ZZ^p\oplus \ZZ^p)$-grading  by setting $\deg(x_{i,j}) = \ee_i \oplus \ee_j \in \ZZ^p \oplus \ZZ^p$.
	As for $\widetilde{R}$, define multidegrees over $R$ in the 
	polynomial ring $\ZZ[\ttt, \sss]$.
	
	\begin{lemma}
		\label{lem_neg_to_pos}
		Let $I \subset \widetilde{R}$ be an $\widetilde{R}$-homogeneous ideal, and denote also by $I$ the corresponding $R$-homogeneous ideal in $R$.
		Then we have 
		$
		\mathcal{C}(\widetilde{R}/I; t_1,\ldots,t_p, s_1,\ldots,s_p)  = \mathcal{C}(R/I; t_1,\ldots,t_p, -s_1,\ldots,-s_p).
		$
	\end{lemma}
	\begin{proof}
		Notice that, if $\widetilde{F}_\bullet$ is a $(\ZZ^p \oplus \ZZ^p)$-graded free $\widetilde{R}$-resolution of $\widetilde{R}/I$ with $\widetilde{F}_i = \bigoplus_{j} \widetilde{R}(-\mathbf{a}_{i,j},-\mathbf{b}_{i,j})$, then there is a corresponding  $(\ZZ^p \oplus \ZZ^p)$-graded free $R$-resolution $F_\bullet$ of $R/I$ with $F_i = \bigoplus_{j} R(-\mathbf{a}_{i,j},\mathbf{b}_{i,j})$.
		By definition, this yields the equality of $K$-polynomials
		$$
		\mathcal{K}(\widetilde{R}/I;\ttt,\sss) = \mathcal{K}(\widetilde{R}/I;t_1,\ldots,t_p,s_1,\ldots,s_p) = \mathcal{K}(R/I;t_1,\ldots,t_p,s_1^{-1},\ldots,s_p^{-1}) = 	\mathcal{K}(R/I;\ttt,\sss^{\mathbf{-1}}).
		$$
		From \cite[Claim 8.54]{miller2005combinatorial}, we have $\mathcal{K}(R/I;\mathbf{1-t},\mathbf{1-s}) = \mathcal{C}(R/I;\ttt,\sss) + Q(\ttt,\sss)$, where $Q(\ttt,\sss)$ is a polynomial with terms of degree at least $\codim(I) + 1$.
		Equivalently, we get $\mathcal{K}(R/I;\ttt,\sss) = \mathcal{C}(R/I;\mathbf{1-t},\mathbf{1-s}) + Q(\mathbf{1-t},\mathbf{1-s})$. 
		It then follows that 
		$$
		\mathcal{K}(\widetilde{R}/I;\mathbf{1-t},\mathbf{1-s}) = \mathcal{C}(R/I; t_1,\ldots,t_p,1-\tfrac{1}{1-s_1},\ldots,1-\tfrac{1}{1-s_p}) + Q(t_1,\ldots,t_p,1-\tfrac{1}{1-s_1},\ldots,1-\tfrac{1}{1-s_p}).
		$$
		By expanding the right hand side of the above equality, the result of the lemma is obtained.	
	\end{proof}

	\section{Standardization of ideals}
	\label{sec:main}
	In this section, we develop a process of standardization of ideals in a certain non-standard multigrading. 
	This process will allow us to show that the support of the multidegree polynomial of any Cohen-Macaulay prime ideal is a discrete polymatroid in the non-standard multigradings that we consider.
	The following setup is used throughout this section.
	
	\begin{setup}
		\label{setup_standardization}
		Let $p \ge 1$ be a positive integer and $\kk$ be a field.
		Let $R$ and $S$ be the polynomial rings $R=\kk[\xx]$ and $S=\kk[\mathbf{w}, \mathbf{z}]$ over the set of variables $\xx = \{x_{i,j}\}_{1\le i,j\le p}$, $\mathbf{w} = \{w_{i,j}\}_{1\le i,j\le p}$ and $\mathbf{z} = \{z_{i,j}\}_{1\le i,j\le p}$.
		We consider $R$ and $S$ as $(\ZZ^p \oplus \ZZ^p)$-graded rings by setting that 
		$$
		\deg(x_{i,j}) = \ee_i \oplus \ee_j,  \quad \deg(w_{i,j}) = \ee_i \oplus \mathbf{0} \quad \text{and} \quad \deg(z_{i,j}) = \mathbf{0} \oplus \ee_j,
		$$
		where $\ee_i \in \ZZ^p$ denotes the $i$-th elementary vector and $\mathbf{0} \in \ZZ^p$ denotes the zero vector.
		We define the $\kk$-algebra homomorphism
		\begin{equation*}
			\phi: R=\kk[\xx] \longrightarrow S = \kk[\mathbf{w}, \mathbf{z}], \quad
			\phi(x_{i,j}) = w_{i,j}z_{i,j}.
		\end{equation*}
		For an $R$-homogeneous ideal $I \subset R$, we say that the extension  $\phi(I) S$ is the \emph{standardization} of $I$, as $\phi(I) S$ is an $S$-homogeneous ideal in the standard multigraded polynomial ring $S$.
		Let $\ttt = \{t_1,\ldots,t_p\}$ and $\sss = \{s_1,\ldots,s_p\}$ be variables indexing the $(\ZZ^p \oplus \ZZ^p)$-grading, where $t_i$ corresponds with $\ee_i \oplus \mathbf{0} \in \ZZ^p \oplus \ZZ^p$ and $s_i$ corresponds with $\mathbf{0} \oplus \ee_i \in \ZZ^p \oplus \ZZ^p$.
		Given a finitely generated graded $R$-module $M$ and a finitely generated graded $S$-module $N$, by a slight abuse of notation, we consider both multidegrees $\mathcal{C}(M;\ttt,\sss)$ and $\mathcal{C}(N;\ttt,\sss)$ as elements of the same polynomial ring $\ZZ[\ttt, \sss]=\ZZ[t_1,\ldots,t_p,s_1,\ldots,s_p]$.
	\end{setup}
	First, we show some basic properties of the process of standardization. 
	
	\begin{proposition}
		\label{prop_std}
		Assume \autoref{setup_standardization}.
		Let $I \subset R$ be an $R$-homogeneous ideal and $J = \phi(I)S$ be its standardization. 
		Then, the following statements hold:
		\begin{enumerate}[\rm (i)]
			\item $\codim(I) = \codim(J)$.
			\item $\mathcal{C}(R/I;\ttt,\sss) = \mathcal{C}(S/J;\ttt,\sss)$.
			\item If $R/I$ is a Cohen-Macaulay ring, then $S/J$ also is.
			\item Let $>$ be a monomial order on $R$ and $>'$ be a monomial order on $S$ which is compatible with $\phi$ (i.e.,~if $f,g \in R$ with $f > g$, then $\phi(f) >' \phi(g)$).
			Then $\init_{>'}(J) = \phi(\init_{>}(I))S$.
		\end{enumerate}
	\end{proposition}
	\begin{proof}
		Let $T$ be the polynomial ring $T=\kk[\xx,\mathbf{w},\mathbf{z}] \cong R \otimes_\kk S$ with its natural $(\ZZ^p \oplus \ZZ^p)$-grading induced from the ones of $R$ and $S$. 
		We now think of $R$ and $S$ as subrings of $T$.
		Consider the quotient ring $T/IT$ and notice that $\{x_{i,j} - w_{i,j}z_{i,j}\}_{1\le i,j \le p}$ is a regular sequence of homogeneous elements over $T/IT$.
		We also have the following natural isomorphism 
		$$
		\frac{T}{IT + \left(\{x_{i,j} - w_{i,j}z_{i,j}\}_{1\le i,j \le p}\right)} \;\cong\; S/J.
		$$
		As the natural inclusion $R \hookrightarrow T$ is a polynomial extension, we have that $\dim(T/IT) = \dim(R/I) + \dim(S) = \dim(R) + \dim(S) - \codim(I)$ and that $T/IT$ is Cohen-Macaulay when $R/I$ is.
		So, by cutting out with the regular sequence described above, we obtain that 
		$		
		\dim(S/J) = \dim(T/IT) - \dim(R) = \dim(S) - \codim(I)
		$
		and that $S/J$ is Cohen-Macaulay when $T/IT$ is.
		This completes the proofs of parts (i) and (iii).
		
		Let $F_\bullet : \cdots  \xrightarrow{f_2} F_1 \xrightarrow{f_1} F_0$ be a graded free $R$-resolution of $R/I$.
		Since $\{x_{i,j} - w_{i,j}z_{i,j}\}_{1\le i,j \le p}$ is a regular sequence on both $T$ and $T/IT$, it follows that $\Tor_k^T\left(T/IT, T/(\{x_{i,j} - w_{i,j}z_{i,j}\}_{1\le i,j \le p})\right) = 0$ for all $k > 0$, and so $G_\bullet = F_\bullet \otimes_R T/(\{x_{i,j} - w_{i,j}z_{i,j}\}_{1\le i,j \le p})$ provides (up to isomorphism) a graded free $S$-resolution of $S/J$.
		The identification of $G_\bullet$ as a resolution of $S$-modules is the same as $\phi(F_\bullet)$ (more precisely,  $G_\bullet$ has the same shiftings as $F_\bullet$ in the $(\ZZ^p \oplus \ZZ^p)$-grading and the $i$-th differential matrix of $G_\bullet$ is given by the substitution $\phi(f_i)$ of $f_i$).
		Therefore,  by definition, we obtain the equality $\mathcal{C}(R/I;\ttt,\sss) = \mathcal{C}(S/J;\ttt,\sss)$ that shows part (ii).
		
		To show part (iv) we can use Buchberger's algorithm (see, e.g., \cite[Chapter 15]{EISEN_COMM}). 
		Indeed, we can perform essentially the same steps of the algorithm in a set of generators of $I$ and the corresponding set of generators for $J$.
	\end{proof}
	
	The following theorem provides the main result of this section. 
	It shows that the support of the multidegree polynomial is a discrete polymatroid for Cohen-Macaulay prime ideals in $R$.
	The proof is carried out by performing a standardization process that allows us to invoke \autoref{thm_pos_multdeg}.
	
	\begin{theorem}
		\label{thm_standardization}
		Assume \autoref{setup_standardization}.
		Let $I \subset R$ be an $R$-homogeneous Cohen-Macaulay prime ideal.
		Then, the support of the multidegree polynomial  $\mathcal{C}(R/I;\ttt, \sss)$ is a discrete polymatroid.
	\end{theorem}
	\begin{proof}
		Let $\mathcal{L} = \{ (i,j) \mid x_{i,j} \in I \}$ be the set of indices such that the corresponding variable belongs to $I$.
		We consider the polynomial rings 
		$
		R' =  \kk[x_{i,j} \mid (i,j) \not\in \mathcal{L}] \subset R
		$ 
		and
		$
		S' =  \kk[w_{i,j}, z_{i,j} \mid (i,j) \not\in \mathcal{L}] \subset S.
		$
		Let $I' \subset R'$ be the (unique) ideal that satisfies the condition $I = I'R + \left(x_{i,j} \mid (i,j) \in \mathcal{L}\right)$.
		By construction we have that $x_{i,j} \not\in I'$ for all $x_{i,j} \in R'$.
		Since $R/I \cong R'/I'$, it follows that $I'$ is also a Cohen-Macaulay prime ideal.
		
		Let $J' = \phi(I'R)S \cap S'$. 
		For any $w_{i,j}z_{i,j} \in S'$, \autoref{prop_std}(i) and the fact that the corresponding $x_{i,j}$ does not belong to the prime $I'$ imply that 
		$$
		\codim(J'S+w_{i,j}z_{i,j}S) = \codim(I'R + x_{i,j}R) = \codim(I'R) + 1 = \codim(J'S) + 1.
		$$
		By \autoref{prop_std}(iii), $S/J'S$ is Cohen-Macaulay, and so it necessarily follows that $w_{i,j}z_{i,j}$ is a non-zero-divisor over $S/J'S$ for all $w_{i,j}z_{i,j} \in S'$.
		Consequently, we obtain that $S'/J'$ is a domain if and only if $\left(S'/J'\right)_{\prod w_{i,j}z_{i,j}}$ is a domain.
		Let $B = \kk[w_{i,j},z_{i,j},z_{i,j}^{-1} \mid (i,j) \not\in \mathcal{L}]$ and consider the automorphism given by
		$$
		\psi : B \rightarrow B, \quad w_{i,j} \mapsto \frac{w_{i,j}}{z_{i,j}}, \quad  z_{i,j} \mapsto {z_{i,j}}.
		$$
		The ideal $\psi(J'B)$ coincides with the extension of $I'$ in $B$ under the ring homomorphism $R' \rightarrow B, x_{i,j} \mapsto w_{i,j}$, and so it follows that $\psi(J'B)$ and, consequently, $J'B$ are prime ideals.
		We then conclude that $J'$ is a prime ideal.
		
		Since the variables $x_{i,j}$ with indices in $\mathcal{L}$ form a regular sequence over $R/I'R$, we obtain the equation 
		\begin{equation}\label{equation:mink_sum}
			\mathcal{C}(R/I;\ttt, \sss) = \prod_{(i,j) \in \mathcal{L}} (t_i+s_j) \cdot \mathcal{C}(R/I'R;\ttt, \sss) 
		\end{equation}
		(see, e.g., \cite[Exercise 8.12]{miller2005combinatorial}).
		To conclude the proof, it is now sufficient to show that the support of $\mathcal{C}(R/I'R;\ttt, \sss)$ is a discrete polymatroid; indeed, we would obtain that the support of $\mathcal{C}(R/I;\ttt, \sss)$ is a Minkowski sum of a finite number of discrete polymatroids which in turn is also a discrete polymatroid  by  \cite[Corollary 46.2c]{schrijver2003combinatorial}.
		Finally, this condition follows by applying \autoref{thm_pos_multdeg} to the prime ideal $J'S$ and exploiting the equality $\mathcal{C}(R/I'R;\ttt, \sss) = \mathcal{C}(S/J'S;\ttt, \sss)$ from \autoref{prop_std}(ii).
	\end{proof}

	We are now ready to prove the main result of this paper.
	
	\begin{proof}[{\bf Proof of \autoref{thmA}}]
		As we already have all the necessary ingredients, the proof follows straightforwardly by combining \autoref{thm_Schubert}, \autoref{lem_neg_to_pos} and \autoref{thm_standardization}.
	\end{proof}
	
	Furthermore, we determine the defining inequalities of the discrete polymatroids in \autoref{thm_standardization} and, accordingly, in \autoref{thmA}.
	Similarly to \autoref{sect_recap_multdeg}, for any two subsets $\fJ_1, \fJ_2 \subseteq [p]$, we denote by $R_{(\fJ_1,\fJ_2)} \subseteq R$ and $S_{(\fJ_1,\fJ_2)} \subseteq S$ the $(\ZZ^{|\fJ_1|} \oplus \ZZ^{|\fJ_2|})$-graded $\kk$-algebras obtained by restricting to the positions in $\fJ_1$ for the first part $\ZZ^p \oplus \mathbf{0}$ of the grading, and to the ones in $\fJ_2$ for the second part of the grading $\mathbf{0} \oplus \ZZ^p$.
	
	\begin{theorem}
		\label{thm:ineqs}
		Assume \autoref{setup_standardization}.
		Let $I\subset R$ be an $R$-homogeneous Cohen-Macaulay prime ideal.
		Then, we have that the coefficient of $\ft^\fr\fs^\fc = t_1^{r_1}\cdots t_p^{r_p} s_1^{c_1}\cdots s_p^{c_p}$ is nonzero in $\mathcal{C}(R/I;\ttt, \sss)$ if and only if  
		\begin{enumerate}[\rm (i)]
			\item $\sum_{j \in [p]} r_j+ \sum_{j \in [p]} c_j = \codim(I)$
			\item For every $\mathfrak{J}_1,\mathfrak{J}_2\subseteq [p]$, we have that  $\sum_{j\in\mathfrak{J}_1}r_j+\sum_{j\in\mathfrak{J}_2}c_j\geq \codim \big(I_{(\fJ_1,\fJ_2)}\big)$, where $I_{(\fJ_1,\fJ_2)}$ is the contracted ideal $I_{(\fJ_1,\fJ_2)} = I \cap R_{(\fJ_1,\fJ_2)}$.
		\end{enumerate}
	\end{theorem}
	\begin{proof}
		By \autoref{thm_standardization}, we know that the Newton polytope of $\mathcal{C}(R/I;\ttt, \sss)$ is a base polymatroid polytope, and so, under the condition $\sum_{j \in [p]} r_j+ \sum_{j \in [p]} c_j = \codim(I)$, all its defining inequalities are of the form
		\begin{equation}\label{equation:minimum}
			\sum_{j\in\mathfrak{J}_1}r_j+\sum_{j\in\mathfrak{J}_2}c_j \geq C(\fJ_1,\fJ_2),
		\end{equation}
		for some constant $C(\fJ_1,\fJ_2)$ 
		that depends on the subsets $\fJ_1,\fJ_2 \subseteq [p]$.
		We now determine $C(\fJ_1,\fJ_2)$.
		
		We keep the same notation of the proof of \autoref{thm_standardization}, in particular, $I = I'R + \left(x_{i,j} \mid (i,j) \in \mathcal{L}\right)$.
		Equation \autoref{equation:mink_sum} decomposes the Newton polytope of $\mathcal{C}(R/I;\ttt, \sss)$ as the Minkowski sum of the Newton polytopes of $\prod_{(i,j) \in \mathcal{L}} (t_i+s_j)$ and  $\mathcal{C}(S/J'S;\ttt, \sss)$, both of which are also base polymatroid polytopes.
		So, we analyze the minimum of the sum in Equation \autoref{equation:minimum} with the two contributions.
		\begin{enumerate}[(a)]
			\item $\textrm{Newton}(\prod_{(i,j) \in \mathcal{L}} (t_i+s_j))$ is determined by the equality $\sum_{j \in [p]} r_j+ \sum_{j \in [p]} c_j = |\mathcal{L}|$ and the inequalities $\sum_{j\in\mathfrak{J}_1}r_j+\sum_{j\in\mathfrak{J}_2}c_j \ge \big|\{(i,j)\in\mathcal{L} \mid i\in\fJ_1 \text{ and } j\in\fJ_2\}\big|$.
			\smallskip
			\item Due to \autoref{thm_pos_multdeg}, $\textrm{Newton}(\mathcal{C}(S/J'S;\ttt, \sss))$ is determined by the equality $\sum_{j \in [p]} r_j+ \sum_{j \in [p]} c_j = \codim(J'S)$ and the inequalities $\sum_{j\in\mathfrak{J}_1}r_j+\sum_{j\in\mathfrak{J}_2}c_j \ge \codim \big({(J'S)}_{(\fJ_1,\fJ_2)}\big)$, where ${(J'S)}_{(\fJ_1,\fJ_2)}$ is the contracted ideal ${(J'S)}_{(\fJ_1,\fJ_2)} = J'S \cap S_{(\fJ_1,\fJ_2)}$.
		\end{enumerate}
		Notice that \autoref{prop_std}(i) yields the equality
		\[
		\codim \big(I_{(\fJ_1,\fJ_2)}\big) = \codim \big(J_{(\fJ_1,\fJ_2)}\big) = \codim \big({(J'S)}_{(\fJ_1,\fJ_2)}\big) + \big|\{(i,j)\in\mathcal{L} \mid i\in\fJ_1 \textrm{ and } j\in\fJ_2\}\big|.
		\]
		Therefore, since we can split the value of $C(\fJ_1,\fJ_2)$ in terms of the sum of the defining inequalities of the Newton polytopes of $\prod_{(i,j) \in \mathcal{L}} (t_i+s_j)$ and $\mathcal{C}(S/J'S;\ttt, \sss)$, it follows that $C(\fJ_1,\fJ_2) = \codim \big(I_{(\fJ_1,\fJ_2)}\big)$. 
		This concludes the proof of the theorem.
	\end{proof}
	
	Finally, we perform a simple computation out of the six possible permutations in $\mathscr{S}_3$ (see \cite[Examples 15.4, 15.42]{miller2005combinatorial}).
	
	\begin{example}[{$p = 3$ and $\pi = (1, 3, 2)$}]
		The Schubert determinantal ideal and the double Schubert polynomials are given by $I_{132} = \left(x_{1,1}x_{2,2}-x_{1,2}x_{2,1}\right)$ and $\fS_{132} = t_1 + t_2 - s_1 - s_2$.
		The standardization of $I_{132}$ is the ideal $J = \left(w_{1,1}w_{2,2}z_{1,1}z_{2,2}-w_{1,2}w_{2,1}z_{1,2}z_{2,1}\right) \in S$.
		The ideal $J$ is prime and $S$ has a standard $(\ZZ^3 \oplus \ZZ^3)$-grading.
		One can compute that 
		$$
		\mathcal{C}(S/J; \ttt, \sss) = t_1+t_2+s_1+s_2
		$$
		(see \cite[Exercise 8.12]{miller2005combinatorial}, or  just utilize the built-in command \texttt{multidegree} on the computer algebra system \emph{Macaulay2} \cite{M2}).
		Coinciding with the claim of \autoref{thm_pos_multdeg}, the support of $\mathcal{C}(S/J; \ttt, \sss)$ is a discrete polymatroid.
		Notice that $\fS_{132} = t_1 + t_2 - s_1 - s_2 = \mathcal{C}(S/J; \ttt, -\sss)$, as shown by \autoref{lem_neg_to_pos}.
	\end{example}

	\begin{remark}
		From the conjectures stated by Monical, Tokcan and Yong \cite{monical2019newton}, a remaining open one is to show that Grothendieck polynomials also satisfy the SNP property (see \cite[Conjecture 5.5]{monical2019newton}). 
		In \cite{K_POLY_MULT_FREE}, we settled a particular case of this conjecture.
		More precisely, we showed that the support of a Grothendieck polynomial is a generalized polymatroid when the Schubert polynomial is zero-one (see \cite[Theorem B]{K_POLY_MULT_FREE}).
	\end{remark}

	

\end{document}